\documentclass [12pt,a4paper,reqno]{amsart}
\usepackage{mathbbol}
\usepackage{amsmath, amssymb}

\input xy
\xyoption{all}

\newcommand{\etype}[1]{\renewcommand{\labelenumi}{(#1{enumi})}}
\def\eroman{\etype{\roman}}

\def\dispace{\setlength{\itemsep}{2pt}}
\newcommand{\ds}[1]{\ {#1} \ }
\newcommand{\dss}[1]{\quad {#1} \quad }

\def\sm{\setminus}

\def\stc{'}
\def\drc{^{\operatorname{c}}}
\def\brV{\overline{V}}
\def\brW{\overline{W}}
\def\brR{\overline{R}}
\def\brS{\overline{S}}

\def\pSkip{\vskip 1.5mm \noindent}
\def\semiring0{semiring$^{\dagger}$}
\def\one{1}
\def\zero{0}
\def\rone{{\one_R}}
\def\rzero{{\zero_R}}

\def\fzero{{\zero_F}}

\def\vzero{\zero_V}

\def\socd{\operatorname{dsoc}}
\def\Indc{\operatorname{Indc}}

\newtheorem{thm}{Theorem} [section]
\newtheorem*{thm*}{Theorem}
\newtheorem{cor}[thm]{Corollary}
\newtheorem{lem}[thm]{Lemma}

\newtheorem{prop}[thm]{Proposition}

\newtheorem*{claim*} {Claim}
\newtheorem{quest}[thm]{Question}
\newtheorem*{theorem13.5'} {Theorem 13.5$'$}
\newtheorem{acknowledgment*}[thm] {Acknowledgment}

\newtheorem{examples}[thm]{Examples}

    \newtheorem*{remarks*} {Remarks}
 \newtheorem*{remark*}{Remark}
 \newtheorem{defn}[thm]{Definition}

\newtheorem*{notation*} {Notation}

\newtheorem{rem}[thm]{Remark}

\def\N{\mathbb{N}}
\def\Z{\mathbb{Z}}
\def\F{\mathbb{F}}
\def\B{\mathbb{B}}

 \renewcommand{\sectionmark}[1]{}

\begin{document}

\title[Modules lacking zero sums]
{Decompositions of Modules lacking zero sums} %$^{*}$}\thanks{$^{*}$ Files for a
%lecture with the same title by the second author at the ``5th
%Linear Algebra Workshop", Kranjska Gora, Slovenia, May 27--June~5,
%2008, subsection ``Real Algebra and its Interactions with
%Functional Analysis"}
\author[Z. Izhakian]{Zur Izhakian}
\address{Institute  of Mathematics,
 University of Aberdeen, AB24 3UE,
Aberdeen,  UK.}
    \email{zzur@abdn.ac.uk; zzur@math.biu.ac.il}
\author[M. Knebusch]{Manfred Knebusch}
\address{Department of Mathematics,
NWF-I Mathematik, Universit\"at Regensburg 93040 Regensburg,
Germany} \email{Manfred.Knebusch@mathematik.uni-regensburg.de}
\author[L. Rowen]{Louis Rowen}
 \address{Department of Mathematics,
 Bar-Ilan University, 52900 Ramat-Gan, Israel}
 \email{rowen@math.biu.ac.il}

\thanks{The third author thanks the University of Virginia
mathematics department for its hospitality.}

%******************************* AMS classification ***********************
%\subjclass{??? Primary 12K10, 13B25; Secondary 51M20}
%******************************* date *************************************
\date{\today}

%******************************* keywords *********************************


%******************************* file name *********************************

\thanks{\noindent \underline{\hskip 3cm } \\ File name: \jobname}

%******************************* abstract *********************************

\begin{abstract} A direct sum decomposition theory is developed for direct
summands (and complements) of modules over a semiring $R$, having
the property that $v+w = 0$ implies $v = 0$ and $w = 0$. Although
this never occurs when $R$ is a ring, it  always does  holds for
free modules over the max-plus semiring and related semirings. In
such situations,  the direct complement is unique, and the
decomposition is unique up to refinement. Thus, every finitely
generated projective module is a finite direct sum of summands of
$R$ (assuming the mild assumption that $1$ is a finite sum of
orthogonal primitive idempotents of $R$). Some of the results are
presented more generally for weak complements and semidirect
complements. We conclude by examining the obstruction to the ``upper
bound'' property in this context.
\end{abstract}

\maketitle

%{ \small \tableofcontents}

%\baselineskip 14pt

\numberwithin{equation}{section}
\section{Introduction}

The motivation of this research is to understand direct sum
decompositions of submodules of free modules  over the max-plus
algebra and related structures in tropical algebra (supertropical
algebra \cite{CC, IzhakianKnebuschRowen2010Linear,IR1}  and
symmetrized algebra \cite{AGG}), as well as in some other settings
in algebra. It turns out that direct sum decompositions are unique
(not just up to isomorphism), and thus
 one can develop a
 theory of direct sum decompositions  analogous to the
theory of the socle in customary abstract algebra, using the  axiom
of ``\textbf{lacking zero sums}'':
$$ v + w = 0 \dss\Rightarrow v=  w = 0$$
(termed ``zerosumfree'' in \cite{golan92}.)
This axiom may
 seem  rather peculiar at first glance, but is easily seen
to hold in tropical mathematics and also over other semirings of
interest, as noted in Examples~\ref{emp:1.6}, especially in real
algebra, such as the positive cone of an ordered field
\cite[p.~18]{BCR} or a partially ordered commutative ring
\cite[p.~32]{Br}. More instances are given in Examples~\ref{emp:1.6}.

After writing the first draft of this paper, we became aware of
\cite{Mac}, in which Macpherson already has proved the uniqueness of
direct sum decompositions of projective modules in the tropical
setting in \cite[Theorem~3.9 and Corollary~4.13]{Mac}, working over
the Boolean semifield~$\mathbb B$. (Then he goes on to prove other
interesting results about projective modules). However, our
hypotheses are different, based solely on this axiom of ``lacking
zero sums,'' which is in the language of elementary logic is a
  quasi-identity (with all its formal implications) and our
main tool (Theorem~\ref{mainthm}) is somewhat stronger than the
decomposition property for projective modules (to be compared with
\cite[Aside~3.7]{Mac}). Also, our main results hold for ``weak
complements,'' which are more general than complements in direct
sums.

 \begin{defn}\label{weakc1} A submodule $T\subset V$ is a \textbf{weak complement} (of a
submodule $W\subset V$)  if $$\text{$T+W  = V$ \quad and \quad $
(w + T) \cap T = \emptyset,$ \  $ \forall \ w \in W \setminus \{
\vzero \} .$}$$
\end{defn}

 Our main
theorems for modules lacking zero sums:

\medskip

\noindent{\bf Theorem~\ref{mainthm}}  Suppose   $V$ has a submodule
$T$ of $V$ which is a weak complement. Then any decomposition of $V$
descends to a decomposition of $T$, in the sense that if  $V = Y +
Z,$ then $T = (T\cap Y) + (T \cap Z).$

\medskip

\noindent{\bf Theorem~\ref{mainthm2}} Suppose    $V = T \oplus W = Y
\oplus Z.$ Then $$V = (T\cap Y) \oplus (T\cap Z) \oplus (W \cap Y)
\oplus (W \cap Z).$$

\medskip

\noindent{\bf Theorem~\ref{mainthm21}} Given two decompositions $V =
T \oplus W = Y \oplus Z$ of  $V$, then
\begin{equation}   T + Y = (T \cap Y) \oplus (T \cap Z)
\oplus (W \cap Y).
\end{equation}

\medskip

\noindent{\bf Theorem~\ref{mainthm3}} Any indecomposable projective
$R$-module $P$ is isomorphic to a direct summand of $R$, and thus
has the form $Re$ for some primitive idempotent $e$ of $R$.

\medskip

For $R$ lacking zero sums, given a module, we start to decompose it,
   and either  we can continue ad
infinitum or the process terminates at an indecomposable summand.
The direct sum of all the indecomposable summands is called the
\textbf{decomposition socle}, denoted $\socd (V)$, and contains
every indecomposable summand of~ $V$, in analogy to the socle (the
sum of the simple submodules) in classical module theory. In fact,
this situation is even tighter than with the classical socle, since
$\socd (V)$ now is written {\it uniquely} as a direct sum of
indecomposables. Furthermore, under certain conditions, e.g., $R =
\N_0,$ the set-theoretic complement of $\socd(V)$, with $\{
\vzero \}$ adjoined,  also is a submodule of $V$ whose intersection
with $\socd(V)$ obviously is $\{ \vzero\}$. (But this is not the
direct complement!)
%
% This leads us to the
%``decomposability socle'' $\socd(R)$ of $R$, consisting of the
%direct sum of the indecomposable projectives.
Furthermore we can understand $\socd(R)$ in terms of the idempotents
of $R$. This approach yields a result analogous to the relation of
semisimplicity and the socle in classical ring theory:

\medskip

 \noindent{\bf Theorem~\ref{mainthm4}}
$\socd(R) = R$ iff $R$ has a finite set of  primitive orthogonal idempotents
whose sum is~ $\rone$, iff $R$ is a finite direct sum of
indecomposable projective modules.

\bigskip

%In particular, all primitive idempotents of $R$ are orthogonal, and
%$R = \socd(R)$ iff $\rone$ can be written as a finite sum of
%primitive idempotents (which means that $R$ is a finite direct sum
%of indecomposable projectives).

We can improve these results by strengthening the lacking zero sum
hypothesis.

\begin{defn}\label{Basdef0} A subset $W$ of a  monoid $(V, +, \vzero)$
is  \textbf{summand absorbing} (abbreviated  ~SA) \textbf{in $V$},
if
\[ \forall \, x, y \in V: \quad x + y \in W \quad \Rightarrow \quad x \in W, y \in W. \]
\end{defn}
An analogous argument yields:
\medskip

 \noindent{\bf Theorem~\ref{2.51}}
Assume that $V = W + T$, where $T$ is SA in $V$, and $W \cap T = \{
\vzero \}$.
\begin{enumerate} \dispace \eroman
\item Then $T$ is the unique weak complement of $W$ in $V$.
\item  If in addition $U$ is a submodule of $V$ with $W + U = V$ and also $U$ is SA in $V$,
then $T \subseteq U$, and $T$ is the unique weak complement of $W
\cap U$ in $U$.
\end{enumerate}

\bigskip

Further along section  \S\ref{weakc} we extend some of these results
to more general decompositions, arising from \textbf{weak
complements} and \textbf{semidirect complements}
(Definition~\ref{semidc}). These would all be the same for modules
over rings, but have subtle distinctions in this setting.

\medskip

 \noindent{\bf Proposition~\ref{sum1}} If $U : = W + S = W \ltimes S$ and $V = U
\ltimes T$, then
\begin{align*}
  S + T &  = S \ltimes T,\\
  W + T & = W \ltimes T, \\
  V = W \ltimes (S \ltimes T) & = (W \ltimes S)
\ltimes T.
\end{align*}
(Here the sign $\ltimes$ denotes  a semidirect decomposition.)
\medskip

\noindent{\bf Proposition~\ref{2.8}} Let $W, S, T$ be submodules of
an $R$-module $V$, and assume that $S$ is a weak complement of $W$
in $U : = W + S$, while $T$ is a semidirect complement of $U$ in
$V$. Then $S + T$ is a weak complement of $W$ in $V$.

\medskip

Finally, one could recall that the tropical situation often involves
the stronger  condition (than lacking zero sums), called
\textbf{upper bound} (ub) that $a + b +c = a$ implies $a+b = a$.
This leads us in \S\ref{ubobs} to utilize Green's partial preorder
 on a semigroup $(V, +)$ by saying that $x \preceq y$ if $x+z = y$ for some $z$ in $V$.
 This yields a
 congruence, the obstruction for a module to be ub, which is studied in terms of a convexity condition, and given in the context of the earlier results
 of this paper.

\subsection{Background}$ $ \pSkip

 We recall that a \textbf{semiring}, denoted
in this paper as a
 $(R,+,\cdot \; , \rzero,\rone)$, is a set $R$
equipped with two binary operations $+$ and~$\cdot \;$, called
addition and multiplication, such that:
\begin{enumerate} \eroman
    \item $(R, +, \rzero)$ is an abelian monoid with identity element $\rzero$; \pSkip
    \item $(R, \cdot \; , \rone )$ is a monoid with identity element
    $\rone$; \pSkip
    \item multiplication distributes over addition. \pSkip
\end{enumerate}
 Modules over
semirings (often called ``semimodules'' in the literature,
cf.~\cite{Dub}) are defined just as modules over rings, except that
now the additive structure is that of a semigroup instead of a
group. (Note that subtraction does not enter into the other axioms
of a module over a ring.) To wit:
%Let us state this explicitly, for the
%reader's convenience.

\begin{defn}\label{def:module0} Suppose $R$ is a semiring.
A (left) \textbf{$R$-module} $V$  is a monoid $(V,+,\vzero)$ together
with scalar multiplication $R\times V \to V$ satisfying the
following properties for all $r_i \in R$ and $v,w \in V$:
\begin{enumerate} \eroman \dispace
    \item $r(v + w) = rv + rw;$
    \item $(r_1+r_2)v = r_1v + r_2 v;$
    \item $(r_1r_2)v = r_1 (r_2 v);$
    \item $\one_R v =v;$
    \item $\zero_R v = r \zero_V = \zero_V.$
  \end{enumerate}
 \end{defn}

We are concerned with the following property.

\begin{defn}\label{Basdef1}
An additive monoid $(V,+,\vzero)$ \textbf{lacks zero sums} if
$v_1+v_2 = \vzero$ implies $v_1 = v_2 = \vzero,$ for any $v_1, v_2
\in V$.
\end{defn}

 Although this condition
never holds when $V$ is a group, since we could take $v_2 = -v_1$,
it always holds in $R^n$ when $R$ is one of the semirings mentioned
in Examples~\ref{emp:1.6} below. In such situations the condition of
lacking zero sums actually is rather ubiquitous, being a
``quasi-identity'' in the language of elementary logic.

\begin{examples} $ $ \begin{enumerate}

\item[a)] If $(V_i \ds \vert i \in I)$ is a family of $R$-modules which  lack zero
sums, then $V := \prod\limits_{i \in I} V_i$ lacks zero sums.
\item[b)] If an $R$-module $V$  lacks zero sums, then the same holds for every submodule of
$V$. In particular, over any semiring  $R$ lacking zero sums, every
submodule of $R^n$ lacks zero sums.
\item[c)] If  $V$ lacks zero sums, then for any set $S$ the  module $\operatorname{Fun} (S, V)$ of functions from $S$ to $V$,
 lacks zero sums.
                    \end{enumerate}
\end{examples}

Thus, any semiring lacking zero sums supports a wide range of
modules lacking zero sums.  If $V= R$ then lacking zero sums is
precisely the condition of being  an ``antiring'' in the sense of
Tan \cite{Tan} and  Dol\u{z}na-Oblak \cite{Dol}, and we have the
following basic examples:
\begin{examples}\label{emp:1.6} $ $
\begin{itemize} \dispace
\item[a)] Obviously if $R \setminus \{ \rzero\}$ is closed under addition
then $R$ lacks zero sums. This happens for the max-plus algebra, the
supertropical algebra mentioned above, and the more general layered
version \cite{IzhakianKnebuschRowenlay} when the ``sorting set'' is
non-negative.  Other instances of this phenomenon worth explicit
mention:
\begin{enumerate}
\item[1)] The ``boolean semifield'' $\mathbb B = \{ -\infty, 0 \}$
(and thus subalgebras of algebras that are free modules over
$\B$).  This shows that our results pertain to ``$\F_1$-geometry,'' treated in \cite{Mac}.

\item[2)] Rewriting the boolean semifield instead as $\B = \{
0, 1\}$ where $1+1 = 1,$ one can generalize it to $\{0, 1,\dots,q\}$
$L=[1,q]:=\{1,2,\dots,q\}$ the ``truncated \semiring0'' of
\cite[Example~2.14]{IzhakianKnebuschRowenlay}, where $a+b$ is
defined to be the minimum of their numerical sum and $q.$

\item[3)] Function semirings, polynomial semirings, and Laurent polynomial
semirings over these semirings.

\item[4)] If $F$ is a formally real field, i.e. $-1$ is not a sum of squares in $F$, then the subsemiring $R = \Sigma F^2$, consisting of all sums of squares in $F$, lacks zero sums. In fact $R$ is a semifield; the inverse of a sum of squares
    \[ a = x_1^2 + \dots + x_r^2 \quad \mbox{is} \quad a^{-1} = \left( {x_1 \over a} \right) + \dots + \left( {x_r \over a} \right)^2. \]

\item[5)] Let $\Z [t] = \Z [t_1, \dots, t_n ]$ denote the polynomial ring in $n$ variables over $\Z$. We choose a non-constant polynomial $f \in \Z [t]$. Then the smallest subsemiring of $\Z [t]$ containing~$f$, namely
    \[ \N_0 [f] = \N_0 + \N_0 f + \N_0 f^2 + \dots \]
lacks zero sums, since $\N_0 [f]$ is a free $\N_0$-module.
\item[6)] More generally, the set of positive elements of any partially ordered
semiring is a sub-semiring lacking zero sums.

\item[7)] The set of finite dimensional characters over a field of characteristic 0 of any group is a semiring lacking zero sums.

\end{enumerate}

\item[b)] Any abelian monoid $(V, +, \vzero)$ can be viewed as a module over
the semiring $\N_0 : = \N \cup \{ 0 \}$, which lacks zero sums.
\end{itemize}
\end{examples}

\begin{defn}\label{Basdef} Suppose that $T$ is a submodule of a module $V$.
We write $T\stc$ for $ V \setminus T$, the \textbf{set-theoretic
complement} of $T$. On the other hand, we define the \textbf{direct
sum} $ T \oplus W $ in the usual way (as the Cartesian product, with
componentwise
 operations).

  A submodule $W\subset V$ is a \textbf{direct complement} of
$T$ if $T\oplus W = V.$
\end{defn}
%
%One emphasis of this paper is on decompositions into direct sums,
%which we call \textbf{decompositions} for short. As in \cite{Mac},
%we shall see in Theorem~\ref{mainthm2} that if $V$ lacks zero sums
%then decompositions  are unique, in the sense that any two finite
%decompositions refine to a common decomposition.

\section{Direct sum decompositions of modules lacking zero sums}

 {\it  We assume through the end of \S\ref{Projmod}  that the $R$-module $V$
lacks zero sums.}

\medskip

Suppose that $W$ and $T$ are submodules of $V$ with $W + T = V$ and  $W\cap T = \{\vzero\}$. In order for $T$ to be a direct
complement of $W$ we need the stronger condition that $w_1+t_1 =
w_2+t_2$ implies $w_1 = w_2$ and $t_1 = t_2.$

The   notion of weak complement (Definition~\ref{weakc1}) goes half
way.

\begin{rem}\label{weakh}
$T\subset V$ is a  weak complement of   $W,$ iff $w_1+t_1 = t_2$
implies $w_1 = \vzero$. In particular $W\cap T = \{\vzero\}$, seen
by taking $t_1 = \vzero.$
\end{rem}

 We turn to the main computation of this paper.

\begin{lem}\label{easy} If $W$ is a  submodule of $V$ with weak complement $T$, and $a = a' + (w_1 + w_2)
$ for $a, a' \in T$ and $w_i \in W$, then $w_1 = w_2 = \vzero$ and $a =
a'.$
\end{lem}
 \begin{proof} By hypothesis $w_1 + w_2 = \vzero,$ implying $w_1 = w_2 =
 \vzero$ since $W$ lacks zero sums.
\end{proof}
\begin{thm}\label{mainthm} Suppose   $V$ has a submodule $T$ of $V$ which is a weak complement. Then any
decomposition of $V$ descends to a decomposition of $T$, in the
sense that if  $V = Y + Z,$ then $T = (T\cap Y) + (T \cap Z).$
\end{thm}
 \begin{proof}   Namely, take a submodule $W$ of $V$ having weak complement $T$, and for
 any $a \in T,$ write $a = y+z$ for $y\in Y$ and $z\in Z.$ In turn
 $y = t_1 + w_1$ and $z = t_2 +w_2$ for $t_i \in T$ and $w_i \in W.$
 Hence $a = (t_1 + t_2) + (w_1 + w_2) \in T + (w_1+w_2).$ By Lemma~\ref{easy}, $w_1 = w_2 = \vzero.$
 Thus $y = t_1 \in T\cap Y$ and  $z = t_2 \in T \cap Z$.
\end{proof}

\begin{cor}\label{2.2} Assume that $W$ is a submodule of $V$. Assume furthermore that $T$ is a weak complement of $W$
in $V$ and $U$ is a submodule of $V$ with $W + U = V$. Then $T
\subset U$.
\end{cor}

\begin{proof} Taking $Y=W$ and $Z = U$ in Theorem~\ref{mainthm},
implies $T = T\cap U$ since  $T\cap W = \{ \vzero \}.$
\end{proof}

\begin{cor}\label{2.3} Any submodules $W$ of $V$  lacking   zero
sums has at most one weak complement in $V$.
\end{cor}
\begin{proof} If $T$ and $U$ are weak complements of $W$ in $V$, then
$T \subset U$ by the theorem. Also $U \subset T$ by symmetry, whence
$T = U$.\end{proof}

%Note that Theorem~\ref{2.2} (or its Corollary~\ref{2.3}) encompasses
%Theorem~\ref{1.1}  (uniqueness of direct complements)% and moreover
%reveals that in Theorem 1.1. it is only necessary to assume that $W$
%(instead of $V$) is free of zero sums. (The proof is the same!)

We leave further results about weak complements to \S\ref{weakc} and
turn more specifically to direct  complements.

Since direct complements are also weak complements, we may state in consequence of
Corollary~\ref{2.3} the following.
\begin{cor}\label{2.6} Given   submodules $T, W,Z$ of $V$, with  $W = W \oplus T= W \oplus Z$, then $T = Z.$
\end{cor}

From Theorem~\ref{mainthm} we draw   the following conclusion.

\begin{thm}\label{mainthm2} Suppose    $V = T \oplus W = Y \oplus Z.$
Then $$V = (T\cap Y) \oplus (T\cap Z) \oplus (W \cap Y) \oplus (W
\cap Z).$$
\end{thm}
\begin{proof} By Theorem~\ref{mainthm}, $T = (T\cap Y) \oplus (T \cap
Z),$ and, symmetrically, $W = (W\cap Y) \oplus (W \cap Z).$ We get
the assertion by putting these together.
\end{proof}

We note in passing that Theorem~\ref{mainthm} also leads to a second
proof of Corollary~\ref{2.6} as follows:

%\begin{cor} Direct complements are unique, in the sense that if $V = T \oplus W = T \oplus Z,$
%then $W = Z.$\end{cor}
\begin{proof}[Second proof of Corollary~\ref{2.6}] Applying  Theorem~\ref{mainthm} to $W$ instead of
$T$, we have $$W = (W\cap T) \oplus (W\cap Z) = W\cap Z$$ since
$W\cap T =\vzero.$ Hence $W \subset Z,$ and, by symmetry,
$Z\subset W,$ yielding $Z = W.$
\end{proof}

Thus any direct summand  $T$ of $V$ has a unique direct complement,
which we denote as~$T\drc$. Note this is properly contained in
$T\stc$ whenever $T \ne \{\vzero\}$, $T \neq V,$   since taking $a
\notin \vzero$ in~$T$ and  $b \notin \vzero$ in~ $T\stc$ we have
$a+b \in T\stc \setminus T\drc.$

\begin{cor} \label{domp1} If $T \subset Y$ then $Y\drc \subset T\drc$.\end{cor}
\begin{proof} Easily seen by refining the decompositions.
\end{proof}

\begin{thm}\label{mainthm21} Given two   decompositions $V = T
\oplus W = Y \oplus Z$ of  $V$, then
\begin{equation}   T + Y = (T \cap Y) \oplus (T \cap Z)
\oplus (W \cap Y).
\end{equation}

\end{thm}

\begin{proof} Write $V = (T \cap Y) \oplus (T \cap Z)
\oplus (W \cap Y)\oplus (W \cap Z),$ and let $\pi$ be the
  projection of $V$ onto  $W \cap Z$ sending $(T \cap Y) \oplus (T
\cap Z) \oplus (W \cap Y)$ to $\vzero$.   Clearly $\pi^{-1} (\vzero)
\subset T + Y$. On the other hand, Theorem~\ref{mainthm} applied
to the decomposition of $T$
 tells us that $T \subset \pi^{-1} (\vzero)$. By symmetry also $Y
\subset \pi^{-1} (\vzero)$, and thus $T + Y \subset \pi^{-1}
(\vzero)$. This proves that
\[ T + Y = \pi^{-1} (\vzero) = (T \cap Y) \oplus (T \cap Z) \oplus (W \cap Y). \]

\end{proof}

\begin{cor}\label{lacki}

If $T$ and $Y$ are direct summands of an $R$-module $V$ lacking zero
sums, then both $T \cap Y$ and $T + Y$ are direct summands of $V$
and
\begin{align}
   (T \cap Y)\drc & = (T \cap Y\drc) \oplus (T\drc \cap
Y) \oplus (T\drc \cap Y\drc),
\\    (T + Y)\drc & = T\drc \cap Y\drc. \label{eq:2.3.3} \end{align} \end{cor}

\begin{prop}\label{complem} Assume that $(U_i \ds \vert i \in I)$ is a
finite family of direct summands of an $R$-module $V$ lacks zero
sums, with decompositions $V = U_i \oplus  U_i\drc$. For any $J
\subset I$ define $U_J : = \bigcap\limits_{j \in J} U_j$ and  $U_J^* : = \bigcap\limits_{j \in J} U_j\drc$. Then

\begin{equation*}  V = \bigoplus\limits_{J \subset I} \big(U_J \cap {U_{I
\setminus J}^*}\big). \end{equation*}

 \end{prop}
\begin{proof}  An easy induction on $\vert I \vert$ starting from
Theorem~\ref{mainthm21}, where we peel off one $U_i$ at a time. %In
%other words, for any given $i$, we have  $\bigcap _{j \ne i} U_j\drc $ is
%the direct sum of the $\bigoplus {U_J}\drc   $ and its direct
%complement in $\bigcap _{j \ne i} U_j\drc $, and we put these together.
\end{proof}

\begin{defn} As usual, we call an $R$-module $V$
\textbf{indecomposable}, if $V \ne \{ \vzero\}$ and $V$ has \textit{no}
decomposition $V = W_1 \oplus W_2$ with $W_1 \ne \{ \vzero \}$, $W_2 \ne \{\vzero\}$.
\end{defn}

Let us turn to the indecomposable direct summands of $V$.

\begin{lem}\label{sum2} If $T$ and $Y$ are indecomposable direct summands of
$V$, then either $T = Y$ or  $T+Y \cong T \oplus Y$.
\end{lem}
\begin{proof}
We obtain from $V  = Y \oplus Y\drc$
by Theorem \ref{mainthm} that $T = (T \cap Y) \oplus (T \cap Y\drc)$, and then,
since $T$ is indecomposable, that $T \cap Y = T$ or $T \cap Y = \{ \vzero \} $, i.e.,
$T \subset Y$ or $T \cap Y =  \{ \vzero \}$.

If $T \subset Y$ we conclude from $V = T \cap T \drc$ in the  same way that
$Y = T \oplus (T\drc \cap Y)$, and then that $Y =T $, since $Y$ is indecomposable.
 If $T \cap Y = \{ \vzero \} $ we have from the above that $T = T \cap Y\drc$, i.e., $T \subset Y\drc$, and now infer from $V = Y \oplus Y\drc$ that $Y + T = Y \oplus T$.
%
%
% Write $V = Y \oplus Z,$ and then note that $T = (T\cap
%Y) \oplus (T\cap Z)$, contrary to $T$   indecomposable unless    $T
%\subset Z,$ or $T \subset Y,$ the latter implying $T = Y.$
\end{proof}

\begin{prop}\label{soc1} The indecomposable direct summands of $V$ are
independent, in the sense that if $T$ and  $ \{T_i: i \in I\}$ are distinct
indecomposable direct summands of $V$, then $$T \cap \big(\sum _i T_i \big) =
\{ \vzero \} .$$
\end{prop} \begin{proof} Taking direct limits, we may assume that $I$ is finite,
and then we are done by Theorem~\ref{mainthm21} and induction.
\end{proof}

\begin{defn} The  \textbf{decomposition socle} $\socd(V)$ is the sum of
the indecomposable direct summands of $V$.
\end{defn}

Now let $\{ T_i : i \in I \}$ denote the set of all indecomposable
direct summands of $V$.

\begin{prop}\label{mainthm22}  When $I$ is finite,
$$\socd(V) = \sum _{i \in I} T_i =  \bigoplus _{i \in I} T_i ,$$
and is a direct summand of $V$, with direct complement $\bigcap
\limits_{i \in I} T_i\drc.$
\end{prop}
\begin{proof}  We may assume that $I = \{ 1, \dots, n \}.$
Let $V_r = \sum_{i = 1}^r T_i$, for $r \leq n$. By an easy
induction, we obtain from Corollary \ref{lacki} that every $V_r$ is
a direct summand of $V$, written
\begin{equation}\label{eq:2.3}
  V = V_r \oplus W_r.
\end{equation}
Furthermore, from Proposition  \ref{soc1}
\begin{equation}\label{eq:2.4}
   V_r \cap T_{r+1} = \{ \vzero\}.
\end{equation}
By Theorem \ref{mainthm} we conclude that
$$T_{r+1} = (V_{r} \cap T_{r+1} ) + (W_{r} \cap T_{r+1} ) = W_{r} \cap T_{r+1},$$
i.e.,
\begin{equation}\label{eq:2.5}
   T_{r+1} \subset W_r.
\end{equation}

Given elements $u,u' \in V_r$, $t,t' \in W_r$ with $u+t = u' +t'$ it follows from
\eqref{eq:2.3} and  \eqref{eq:2.5} that $u = u'$ and $t =t'$. Thus
$$ V_{r+1} = V_r \oplus T_{r+1}$$
for every $r < n$. The proposition now follows, up to the last
assertion, which can be obtained from \eqref{eq:2.3.3} in Corollary
\ref{lacki} by another easy induction.
%
%By Lemma~\ref{sum2}, $T_1 + T_2 =  T_1 \oplus T_2,$ and thus, by
%induction, $$T_1 + T_2 + \cdots + T_i =  T_1 \oplus T_2 \oplus
%\cdots \oplus T_i $$ for all $1 \le i \le n.$ Furthermore,  by
%induction on Corollary~\ref{lacki}, $T_1 + T_2 + \cdots + T_i $ has
%complement $T_1^c \cap T_2^c\cap \cdots \cap T_i ^c$ for all $i \le
%n.$
\end{proof}

  For $I$   infinite, it is seen in the same way that $\socd(V)$ is
  the direct sum of the $T_i$, but now it need not be a  direct summand of
  $V$. Furthermore, we must cope with the possibility that  $$\socd(V)'_0 : =  \socd(V) \cup \{ \vzero \}$$
  is not  an $R$-submodule of $V$, but just a submonoid.

 To rectify the situation, we view $V$ as an $\N_0$-module, and let $$\Indc(V): =\{ W_i : i \in I '\}$$
 denote the set of all indecomposable
$\N_0$-direct summands of~$V$, and
$$W:= \sum _{i \in I'} W_i =  \bigoplus _{i \in I'} W_i .$$
Then $$W_0' := (V\setminus W) \cup \{ \vzero \}$$
 does not contain any
$\N_0$-indecomposable direct summands of $V$, and in particular none
of the ~$T_i$.

Let $R^\times$ denote the group of units of $R$. Every $\lambda \in
R^\times$ yields an automorphism $v \mapsto \lambda v$ of $(V, +,
\vzero),$ and so the group $R^\times$ operates on  $\Indc(V)$.  When
the semiring $R$ is additively generated by $R^\times$, the $T_i$
are precisely the sums $\sum _{i \in J} W_i =  \bigoplus _{i \in J}
W_i $ where $\{ W_i : i \in J\}$ is an orbit of $R^\times$ on
$\Indc(V).$ The following result follows immediately from these
observations.

\begin{thm}\label{mainthm6} Assume that the semiring $R$ is
additively generated by $R^\times$. Then $\socd(V)$ is the direct
sum  of all indecomposable direct summands of $V$, and the additive
monoid $$\socd(V)_0' := (V\setminus \socd(V)) \cup \{ \vzero \}$$ is an
$R$-submodule of $V$. (But if there are infinitely many such
indecomposable direct summands, $\socd(V)$ need not be a direct
summand of $V$.)\end{thm}

Examples where the theorem applies are:
\begin{itemize} \dispace
  \item $R = \N_0$;
  \item $R$ is a semifield;
  \item $R$ is a so-called \textbf{supersemifield}, i.e., a supertropcial semiring (cf. \cite{IR1}, \cite{IKR3}) where both $R \sm (eR)$ and $(eR) \sm \{ \rzero\}$ are groups, with $e = \rone + \rone$;

  \item $R$ is replaced by
  the semiring $R[t_1, t_1^{-1}, \dots, t_n, t_n^{-1} ]$ of Laurent polynomials in $n$ variables over any of the previous semirings $R$.
\end{itemize}

%Examples of such semirings include $\mathbb N_0,$ semifields,``
%supersemifields'' where $R \setminus eR$ and $eR \setminus \{0 \}$
%ar groups for $e = 1+1,$ and Laurent semirings over these semirings.
% since by definition this
%would be in $\socd(V)$. In view of Proposition~\ref{soc1}, the sum
%is a direct sum, so we conclude with Proposition~\ref{complem}.

\section{Projective $R$-modules}\label{Projmod}

We are ready to apply these results to the case that $V = R^n$.
Assume throughout this section that $R$  lacks zero sums.

\begin{defn} A module $P$ is \textbf{projective} if it is a
direct summand of a free $R$-module, and is \textbf{finitely
generated projective} if it is a summand of $R^n$.

An   element $e \in R$ is \textbf{idempotent} if $e^2 = e.$ Two
idempotents $e,f$ are \textbf{orthogonal} if $ ef = fe = \rzero.$ An
idempotent is \textbf{primitive} if it cannot be written as the sum
of two nonzero orthogonal idempotents.\end{defn}

Projective modules are treated much more generally in \cite{Dur}.
The same argument as in \cite{Mac} yields the analogous conclusion.

\begin{thm}\label{mainthm3} Any indecomposable projective $R$-module $P$ is isomorphic to a direct
summand of $R$, and thus has the form $Re$ for some primitive
idempotent $e$ of $R$. \end{thm}
\begin{proof} Write the free module $F = P \oplus P\drc = \bigoplus \limits_{i\in I}  R \varepsilon _i$
with base $\{ \varepsilon_i : i \in I \}.$ For each $i \in I,$
\begin{enumerate}
  \eroman \dispace
\item $R\varepsilon_i = ((R\varepsilon_i) \cap P)\oplus ((R\varepsilon_i) \cap
P\drc);$
\item $P \varepsilon_i = (R\varepsilon_i \cap P)\oplus \sum\limits _{j \ne i} ((R\varepsilon_j) \cap
P).$\end{enumerate} If $(R\varepsilon_i) \cap P = \{ \fzero\},$ then
(i) yields $(R\varepsilon_i) \cap P\drc = R\varepsilon_i,$ so
$R\varepsilon_i \subset P\drc;$ if this holds for all $i$ then
$P\drc = F$ implying $P = 0.$

Thus we may assume that there is $i \in I$ with $(R\varepsilon_i)
\cap P \ne \{ \fzero \}.$ Now (ii) implies $(R\varepsilon_i) \cap P = P$, since
$P$ is indecomposable, whence $P$ is a direct summand of
$R\varepsilon_i,$ and we may assume that $F = R.$
%
%Since $F$ is a direct sum of copies of $R$, we have $R = (R \cap P)
%\oplus (R \cap P\drc)$. But $R \cap P$ is a direct summand of $P$,
%so if nonzero, we have $R \cap P = P,$ and $P$ is a direct summand
%of $R$. If all of these are 0 then $P\drc = R,$ implying $P= 0$.

 Now consider the projection $\pi : R \twoheadrightarrow P$
onto $P$. Letting $e = \pi (\rone)$ we have $e^2 = e \pi (\rone) = \pi (e) =
e,$ so $e$ is idempotent. If $e $ were not primitive then writing $e
= e_1 + e_2$  for orthogonal idempotents $e_1 , e_2$  would yield
$Re = Re_1 \oplus Re_2$. (The standard proof for modules over rings
also holds here.) This would contradict the indecomposability of
$P$.
\end{proof}

\begin{thm}\label{mainthm4} $\socd(R) = R$ iff $R$ has a finite set
of orthogonal primitive idempotents whose sum is~ $\rone$, iff $R$
is a finite direct sum of indecomposable projective modules.
\end{thm}
\begin{proof} In view of Theorem~\ref{mainthm3}, the only thing
remaining to check here is the finiteness. But this is another
standard argument taken from ring theory. If $R = \bigoplus _i P_i,$
then the unit element~$\rone$ is in this sum, and thus is some
finite sum of elements $\sum_i r_i e_i,$ implying $R$ is the sum of
these~$P_i$.
\end{proof}

\begin{cor}\label{mainthm31} If $R$ has a finite set
of orthogonal primitive idempotents $\{e_1, \dots, e_m\}$ whose sum
is $\rone$, then every finitely generated projective $R$-module $P$
is a finite direct sum of indecomposable projective modules, i.e.,
$P \cong \bigoplus _{i=1}^m (Re_i)^{n_i}$ for suitable $n_i$, and
this direct sum decomposition is unique.\end{cor}

 % We  focus first on weak complements.

\section{Submodules satisfying the summand absorbing property}\label{weakc}

Throughout,  $R$ is a semiring and $V$ an $R$-module.  Perhaps
surprisingly at first glance, uniqueness of decompositions can be
proved in   settings where we drop the requirement that  ~$V$ lacks
zero sums (but strengthen the requirement on $W$).
So we drop this hypothesis and
focus instead on its submodule $W$, first in conjunction with SA
(Definition~\ref{Basdef0}), and then in terms of ``semidirect
complements.''

\begin{prop}\label{2.4} Assume that $W$ is a submodule of $V$ and
that $T$ is a weak complement of~$W$ in $V$. Assume also that $V
\setminus T$ is closed under addition. Then $W$ lacks  zero sums
(and so $T$ is the unique weak complement of $W$ in $V$).\end{prop}
\begin{proof} Let $w_1, w_2 \in W \setminus \{ \vzero  \}$. Then $w_i \not\in T$ for $i = 1,2$, implying $w_1 + w_2 \not\in T$. Thus certainly $w_1 + w_2 \ne
\vzero$. \end{proof}

\begin{lem}\label{eqcond} The following conditions are equivalent for $W\subset V$:
  \begin{enumerate}\eroman \dispace
  \item $W$   is SA in $V$;

    \item  The set-theoretic
complement $ W'$ of $W$ is an additive (monoid) ideal, in the sense
that $w + v \in W'$ for all $w \in W',$ $v \in V$;

 \item If $\sum\limits_{i=1}^m  a_i \in W$, then each $a_i \in W$.
\end{enumerate}\end{lem}
\begin{proof} $(i) \Leftrightarrow (ii)$ is clear, and $(i) \Rightarrow
(iii)$ by induction on $m$. $(iii) \Rightarrow (i)$ is immediate.
\end{proof}

We pass to the case of $R$-modules (which is not much of a
transition, since an additive  monoid is an $\N_0$-module).

\begin{lem} Assume that $\varphi : V_1 \rightarrow V_2$ is a homomorphism
of $R$-modules over an arbitrary semiring  $R$. If $T$ is a
SA-submodule of $V_2$, then $\varphi^{-1} (T)$ is a SA-submodule of
$V_1$.\end{lem}
\begin{proof} Let $x, y \in V$, and assume $x + y \in \varphi^{-1} (T)$. Then
$\varphi (x) + \varphi (y) = \varphi (x+y) \in T$, and so $\varphi
(x), \varphi (y) \in T$, whence $x, y \in \varphi^{-1} (T)$.
\end{proof}

%In the present section we will use this principle only in the
%special case that $V_1 = V_2$, and $T = \{ 0 \}$.
%%, where it reads as
%%follows.

\begin{rem}
$ $
\begin{enumerate} \dispace
  \item[a)] If $(U_i
  \ds\vert i \in I)$ is a family of
SA submodules of an $R$-module $V$, then the intersection
$\bigcap\limits_{i \in I} U_i$ clearly  also is SA in $V$.

 \item[b)] If $(U_i \ds\vert i \in I)$ is an upward directed family of
SA submodules of~ $V$, i.e. for any $i, j \in I$ there exists $k \in
I$ with $U_i \subset U_k$, $U_j \subset U_k$, then the union
$\bigcup\limits_{i \in I} U_i$ is a SA submodule of
$V$.\end{enumerate}
\end{rem}

 %\begin{lem}\label{basic0} If $V$  lacks zero sums, then any direct summand of $V$
% is SA. \end{lem}
% \begin{proof} Suppose $V = T \oplus W$. If $a + v   \in
%T$ for $a \notin T,$ then writing $a = t_1 + w_1$ and $v = t_2 +
%w_2,$ for $t_i \in T$ and $w_i \in W$, then $(t_1 + t_2) + (w_1+w_2)
%\in T,$ implying $w_1+w_2 =\vzero,$ and thus $w_1 = w_2 = \vzero.$
%Hence $a,v \in T.$
%\end{proof}
%
%This basic argument repeats in a different guise in
%Theorem~\ref{mainthm}.
%
%Although our proofs largely are applications of the definition of
%SA, it is more convenient to bypass this definition for the time
%being, and make the blanket assumption that $V$ lacks zero sums.

Note in Definition~\ref{Basdef} that for $R = \N_0,$
Lemma~\ref{eqcond}(ii) implies that any SA submodule $T$ is a weak
complement of $T\stc \cup \{ \vzero \}$.

\begin{thm}\label{2.51} Assume that $V = W + T$, where $T$ is SA in $V$, and
$W \cap T = \{ \vzero  \}$.
\begin{enumerate} \dispace \eroman
\item Then $T$ is the unique weak complement of $W$ in $V$.
\item  If in addition $U$ is a submodule of $V$ with $W + U = V$ and also $U$ is SA in $V$,
then $T \subset U$, and $T$ is the unique weak complement of $W
\cap U$ in $U$.
\end{enumerate}
\end{thm}
\begin{proof} (i): Let $w \in W \setminus \{ \vzero \}$ and $t \in T$.
Suppose that $w + t \in T$. Then $w \in T$, contradicting $W \cap T
= \{ \vzero \}$. Thus $(w + T) \cap T = \emptyset$.
\pSkip

  (ii): $V \setminus T$ is closed under addition. By Proposition
~\ref{2.4} we know that $W$  lacks zero sums and thus $T \subset
U$. Since $U$ is SA in $V$ we conclude from $V = W + T$ that $U = (W
\cap U) + T$. By part (i),   $T$ is the unique weak complement of $W
\cap U$ in $U$, since $T$ is SA in $U$.
\end{proof}

Here is one nice kind of weak complement.

\begin{defn}\label{semidc}  Let $W$ and $T$ be $R$-submodules of $V$.
%\begin{itemize}
%\item   $T$ is a \textbf{weak complement of $W$ in} $V$ if $W + T = V$ and
%
%\begin{equation}\label{2.1} \forall \ w \in W \setminus \{ 0 \} : \quad (w +
%T) \cap T = \emptyset. \end{equation}
%
%\item
$T$ is a \textbf{semidirect complement of~$W$ in} $V$ if $W + T = V$
and
\begin{equation}\label{sprod}\forall  w_1, w_2 \in W : \quad w_1 \ne w_2
\dss\Rightarrow (w_1 + T) \cap (w_2 + T) = \emptyset. \end{equation}
 In this case, we  also say that $V$ is the \textbf{semidirect sum} of $W$ and $T$ and
write
  $V = W \ltimes T.$
\end{defn}

 Condition~\eqref{sprod} can be recast  as follows: For any $w_1, w_2 \in W$,
$t_1, t_2 \in T$,
\begin{equation*}w_1 + t_1 = w_2 + t_2 \dss\Rightarrow w_1 = w_2.\end{equation*}
 This means that there exists an $R$-linear projection  $p : V
\rightarrow V$ given by $p(w+t) = w,$ with image $p(V) = W$ and
kernel $p^{-1} (\vzero) = T$.  We sometimes write
\begin{equation}\label{2.5} p = \pi_{W, T}.\end{equation}

%
% \textit{Comment.} Condition \ref{2.2} can be recast as follows:
%For any $w_1, w_2 \in W$, $t_1, t_2 \in T$,
%
%\begin{equation}\label{2.4} w_1 + t_1 = w_2 + t_2 \Rightarrow w_1 = w_2.\end{equation}
%
%  This means that there exists an $R$-linear projection $p : V
%\rightarrow V$ with image $p(T) = W$ and kernel $p^{-1} (0) = T$.
%The projection $p$ is uniquely determined by $W$ and $T$. We
%sometimes write
%
%\begin{equation}\label{2.5}   p = \pi_{W, T}.
%\end{equation}

In summary, we have the following hierarchy of conditions on modules
$T,W$  satisfying $W+T = V$, each implying the next, which are all
equivalent in classical module theory over a ring:

\begin{enumerate}\eroman \dispace
   \item  $T$ is a
direct complement of $W$;

  \item   $T$  is a semidirect complement of $W$;

    \item  $T$ is a weak complement  of $W$;

   \item    $W \cap T = \{ \vzero\}.$
  \end{enumerate}

Here the reverse implications may fail. We now address
``transitivity'' of these various complements.

%
%\begin{equation}\label{sprod} \forall \ w_1, w_2 \in W : \quad w_1 \ne w_2 \Rightarrow (w_1 +
%T) \cap (w_2 + T) = \emptyset.\end{equation}
%  We then also say that $V$ is the \textbf{semidirect sum} of $W$
%and $T$ and write
%\begin{equation*}V = W \ltimes T.\end{equation*}

\begin{quest} Assume that $W, S, T$ are submodules of an
$R$-module $V$ such that $S$ is a complement of $W$ in $U : = W + S$
of a certain type (direct, semidirect , weak) and $T$ is a
complement of $U$ in $V$ of the respective type. Then is  $S + T$ a
complement of $W$ in $V$,  of this respective type?
\end{quest}

This is obviously true for direct complements. It also holds for
semidirect complements. More explicitly, we have the following
facts.

\begin{prop}\label{sum1} If $U : = W + S = W \ltimes S$ and $V = U
\ltimes T$, then
\begin{align}
  S + T &  = S \ltimes T,\\
  W + T & = W \ltimes T \\
  V = W \ltimes (S \ltimes T) & = (W \ltimes S)
\ltimes T.
\end{align}
 \end{prop}
 We give two proofs of these facts, having different flavors.

\begin{proof}[First proof] Here we use the definition of semidirect complements given in
\eqref{sprod}. Let $w_1, w_2 \in W$ and $w_1 \ne w_2$. Then $(w_1 +
S) \cap (w_2 + S) = \emptyset$ and so $w_1 + s_1 \ne w_2 + s_2$ for
any $s_1, s_2 \in S$. Since $T$ is a semidirect complement of $W +
S$ in $V$ we have in turn
\[ (w_1 + s_1 + T) \cap (w_2 + s_2 + T) = \emptyset. \]
 This proves that
\begin{equation}\label{ast0}
  (w_1 + S + T) \cap (w_2 + S + T) = \emptyset,
\end{equation}
  and it follows that
\[ (w_1 + T) \cap (w_2 + T) = \emptyset. \]

  If $s_1 \ne s_2 $ in $S$  then, since $s_1$ and $s_2$
are different elements of $W + S$, we also conclude from
\eqref{ast0} that $(s_1 + T) \cap (s_2 + T) = \emptyset$.
\end{proof}

\begin{proof}[Second proof] We employ the projections associated to
semidirect decompositions, cf.~\eqref{2.5}, identifying any
projection  $p : X \to X$ onto an $R$-module $X$ with the induced
surjection $X \twoheadrightarrow p (X)$. We have projections $p :=
\pi_{U, T} : V \twoheadrightarrow U$ and $q : = \pi_{W, S} : U
\twoheadrightarrow W$ with respective kernels $T$ and $S$. Then $r :
= q \circ p : V \twoheadrightarrow W$ is a projection  with kernel
$S + T$, yielding $V := W \ltimes (S + T)$. The projection  $r : V
\twoheadrightarrow W$ restricts to maps $r \vert (S + T)
\twoheadrightarrow S$ and   $r \vert (W + T) \twoheadrightarrow T$,
which both are projections with kernel $T$. Thus $S + T = S \ltimes
T$ and $W + T = W \ltimes T$.
\end{proof}

  For weak complements we cannot expect a transitivity statement such as
(4.5) above. But a ``mixed transitivity'' holds for weak and
semidirect complements.

\begin{prop}\label{2.8} Let $W, S, T$ be submodules of an $R$-module
$V$, and assume that $S$ is a weak complement of $W$ in $U : = W +
S$, while $T$ is a semidirect complement of $U$ in $V$. Then $S + T$
is a weak complement of $W$ in $V$. \end{prop}
\begin{proof} Let $w \in W \setminus \{ \vzero \}$ and $s_1, s_2 \in S$.
Then $(w + S) \cap S = \emptyset$. Thus $w + s_1$ and $s_2$ are
different elements of $W + S = U$, which implies that $(w + s_1 + T)
\cap (s_2 + T) = \emptyset$. This proves that $(w + S + T) \cap (S +
T) = \emptyset$, as desired. \end{proof}

We finally mention a result of independent interest, which can be
obtained by a slight amplification of the proof of
Proposition~\ref{sum1}.

\begin{prop}\label{2.9} Assume that $W, T, U$ are submodules of an
$R$-module $V$ with $W + T \subset W + U$ and $W \cap T \subset W
\cap U$. Assume furthermore that $W$ lacks zero sums, and $T$ is SA
in~ $V$. Then $T \subset U$.
\end{prop}
\begin{proof} Let $t \in T$ be given. We write $t = w + u$ with $w
\in W$, $u \in U$. Since $T$ is SA in $V$, this implies that $w \in
T$, whence $w \in W \cap T \subset W \cap U$. We conclude that $t =
w + u \in U$.\end{proof}

\section{The obstruction to the ``upper
bound''    condition}\label{ubobs}

Recall from  \cite{IKR3} that an additive monoid $(V,+,\vzero)$ is
\textbf{upper bound} if $x+y+z =x$ implies $x+y  =x$. This property
instantly implies ``lacking zero sums.''
The object of this section
is to study the obstruction to this condition.

\begin{defn}\label{defub} Define  \textbf{Green's partial preorder}
 on a monoid $(V, +, \vzero )$ by saying $$\text{$x \preceq y$ \ if \  $x+z = y$ for some $z$ in $V$.}$$
We write $x \equiv y$ if $x \preceq y$ and $y \preceq x$.
\end{defn}

Clearly $\preceq $ is reflexive and transitive, implying that
$\equiv$ is an equivalence relation; in fact, $\equiv$ is a
congruence, since if $x \preceq y$ then  $x +a \preceq y+a$ for any
$a\in V$. (Indeed, if $x+z = y$, then $x+a+z = y+a$.) Accordingly,
$\brV : = V/\equiv$ also is a monoid, with the induced operation
$\bar x  + \bar y  = \overline{x+y} ,$ where $\bar x$ denotes the
equivalence class of $x$. $ \preceq $ induces a partial order $\le$
on~$\brV$, given by
\begin{equation} \bar x \le \bar y \quad \text{if} \quad x
\preceq y.
\end{equation}

\begin{lem} The monoid $\brV$ is upper bound.
\end{lem}
\begin{proof} Suppose $\bar x+\bar y+\bar z =\bar x$. Then $  x+  y+  z +z' = x$
for some $z'$, implying $  x+  y \le x$. But clearly $  x\le   x+
y,$ so $  x+y\equiv   x ,$ and  $\bar x+\bar y  =\bar x$.
\end{proof}

This construction respects other topological notions.

\begin{defn}\label{defub} A subset $S\subset V$ is \textbf{convex}
if, for any $s_i$ in $S$ and $v$ in $V$, $s_1 \preceq v  \preceq
s_2$ implies $v \in S.$
\end{defn}

\begin{lem}\label{conhul} The convex hull of a point $s \in S$ is its equivalence class in $\brV$.
\end{lem}
\begin{proof} $(\subset)$ If $s + y + z = s,$ then $s \preceq s+y \preceq s,$ implying $s \equiv s+y.$
\pSkip

$(\supset)$ If  $s \equiv s+y,$ then $s + y + z = s$ for some $z$,
implying $s  \preceq s+y \preceq s.$
\end{proof}

\begin{prop}\label{conv1} $S$ is convex in $V$ iff  $\brS$ is convex in $\brV$ and $S$ is a  union of
equivalence classes.
 \end{prop}
\begin{proof}
Take the convex hull and apply Lemma~\ref{conhul}.
\end{proof}

This also ties in with the SA property.

\begin{lem}\label{conhul1} A subset $S$ containing $\vzero$  is convex in $V$,
iff $S$ is SA.
\end{lem}
\begin{proof}
$(\Rightarrow)$ If $a+b \in S$ then $\vzero \preceq a  \preceq a+b$
implies $a \in S$, and likewise $b \in S.$
\pSkip

$(\Leftarrow)$ If $s _1 \preceq a \preceq s_2,$ then writing   $s_2
= a + z_2,$ we have $a \in S.$
\end{proof}

\begin{prop}  A submodule $S \subset V$ is SA in $V$, iff $S$ is a  union of
equivalence classes and $\brS $ is SA in $\brV$.
 \end{prop}
\begin{proof} This follows from Proposition~\ref{conv1} and
Lemma~\ref{conhul1}, applied to $S$ and $\brS $.
\end{proof}

Some concluding observations:
\begin{rem}
$ $
\begin{enumerate} \eroman \dispace
  \item  If $R$ is a semiring, then (taking $V = R$) the equivalence $\equiv$
  also respects multiplication, so $R/\equiv$ is a  ub~semiring.

 \item  If $V$ is an $R$-module, then $\brV$ is an  $\brR$-module,
 where scalar multiplication is given
 by $\bar a \bar v =
 \overline{av}.$

  \item Any decomposition $V = W_1 \oplus W_2$ induces a decomposition
  $ \brV =  \brW_1 \oplus  \brW_2.$
  \end{enumerate}
\end{rem}


\begin{thebibliography}{IMS}

\bibitem{AGG}
M.~Akian,   S.~Gaubert, and A.~Guterman.
\newblock Linear independence over tropical semirings and beyond.
\newblock In {\em Tropical and  Idempotent Mathematics}, G.L. Litvinov and S.N.
Sergeev, (eds.),
\newblock {\em Contemp. Math.}  495:1--38, 2009. % to appear (Preprint at arXiv:math.AC/0812.3496v1).

\bibitem{Br} G.W.~Brumfiel. {\em Partially ordered   rings and semialgebraic
geometry}, London Math. Soc. Lecture Notes 37, Cambridge Univ.
Press, 1979.

\bibitem{BCR} J.~Bochnak, M.~Coste, and M.-F.~Roy.
 {\em Real algebraic
geometry}, Ergebnisse Math. Grenzgeb.~36, 1998.


\bibitem{CC}
A.~Connes and   C.~Consani.
\newblock Characteristic 1, entropy, and the absolute point.
\newblock  {\em Noncommutative Geometry, Arithmetic, and Related Topics}, {Proc. of the 21st JAMI Conference},
Johns Hopkins Univ. Press, 75--139, 2011.
% \newblock Preprint at arXiv:math.0911.3537, 2009.

\bibitem{Dol} D. Dol\u{z}na and P. Oblak.
\newblock  Invertible and nilpotent matrices over antirings. {\em Lin. Alg. and Appl.}, 430(1):271--278, 2009.


 \bibitem{Dub}
M.~Dubey.
\newblock {\em Some results on semimodules analogous to module theory},
\newblock Doctoral Dissertation, \newblock University of Delhi, 2008.

 \bibitem{Dur}  N. Durov. A new approach to Arakelov geometry, arXiv
0704.2030, 2007.


\bibitem{golan92}
J.~Golan.
\newblock {\em Semirings and their Applications}, Springer-Science + Business, Dordrecht, 1999.
\newblock (Originally published by Kluwer Acad. Publ.,  1999.)

%
%\bibitem[IKR1]{IKR1}
%%57
% Z.~Izhakian, M.~Knebusch, and L.~Rowen.
%\newblock Supertropical semirings and supervaluations.
%\newblock Preprint at arXiv:1003.1101,  2009.
%
%\bibitem{zur05TropicalAlgebra}
%Z.~Izhakian.
%\newblock Tropical arithmetic and matrix algebra.
%\newblock {\em Comm. in Algebra}  {37}(4):1445--1468, {2009}.


\bibitem{IzhakianKnebuschRowen2010Linear}
%77
Z.~Izhakian, M.~Knebusch, and L.~Rowen.
\newblock Supertropical linear algebra.  {\em Pacific J. of Math.} 266(1):43--75, 2013.
%%\newblock (Preprint at arXiv:1008.0025, 2010.)



\bibitem{IzhakianKnebuschRowenlay} Z.~Izhakian, M.~Knebusch, and L.~Rowen.
 Layered tropical mathematics, {\em  J. Algebra} 416:200--273, 2014.


 \bibitem{IR1}
Z.~Izhakian and L.~Rowen.
\newblock Supertropical algebra.
\newblock   {\em Adv. in Math.} 225(4):2222--2286, 2010.
%\newblock (Preprint at arXiv:0806.1175).


\bibitem {IKR3}
Z.~Izhakian, M.~Knebusch, and L.~Rowen.
\newblock Supertropical semirings and supervaluations.
 \newblock  \emph{J. Pure and Appl.~Alg.} 215(10):2431--2463, 2011.

%
%\bibitem[IR2]{IR2}
%Z.~Izhakian and L.~Rowen.
%\newblock \textit{Supertropical matrix algebra},
%\newblock  to appear in Israel J. Math.
%\newblock (Preprint at arXiv:0806.1178, 2008).

 \bibitem {KZ1} M. Knebusch and D. Zhang.  Convexity, valuations, and
Pr\"ufer extensions in real algebra, \textit{Doc. Math.} 10:1--109, 2005.


\bibitem{Mac} A.W. Macpherson, Projective modules over polyhedral
semirings,
 arXiv:1507.07213v1 [math.AC] 26 Jul 2015.

\bibitem{Tan}
Y. Tan. \newblock  On invertible matrices over antirings. {\em Lin.
Alg. and Appl.} 423(2):428--444, 2007.


 \bibitem{Y} S.~Yoshimoto. {Generators of modules in tropical
 geometry},
 \newblock preprint at arXiv:1001.0448v2, 2010.
\end{thebibliography}
\end{document}